\documentclass[reqno, a4paper, 11pt, oneside]{article}
\usepackage[11pt]{extsizes}
\usepackage[a4paper,hmargin=2.5cm,vmargin=2.5cm, bindingoffset=0cm]{geometry}

\usepackage{mathtools}
\usepackage[utf8]{inputenc}
\usepackage[T2A]{fontenc}
\usepackage{amsmath,amssymb,amsthm}
\usepackage[english]{babel}
\usepackage{textcomp}
\usepackage[usenames]{color}
\usepackage{colortbl}
\usepackage{float}
\usepackage{relsize}
\usepackage{amsthm}

\usepackage{thmtools}
\usepackage{thm-restate}

\makeatletter
\newcommand*{\rom}[1]{\expandafter\@slowromancap\romannumeral #1@}
\makeatother

\usepackage[bookmarksopen]{hyperref}
\hypersetup{
	colorlinks=false,
	linkcolor=blue,
	citecolor=magenta,
	urlcolor=cyan,
}
\usepackage{cleveref}

\usepackage{graphicx}
\DeclareGraphicsExtensions{.pdf,.png,.jpg}

\usepackage{amsthm}

\theoremstyle{definition}

\newtheorem*{theorem*}{Theorem}
\newtheorem{theorem}{Theorem}[section]

\newtheorem{lemma}[theorem]{Lemma}
\newtheorem{cor}[theorem]{Corollary}
\newtheorem{Remark}[theorem]{Remark}
\newtheorem{sign}[theorem]{Notation}

\newtheorem{Que}[theorem]{Question}
\newtheorem{proposition}[theorem]{Proposition}
\newtheorem{ex}[theorem]{Example}

\newtheorem{definition}[theorem]{Definition}

\newtheorem*{theorem-non}{Theorem}

\definecolor{applegreen}{rgb}{0.0, 0.42, 0.24}
\definecolor{applegreen}{rgb}{0.55, 0.71, 0.0}

\newcommand{\Fs}{\mathbf{F}}

\newcommand{\tl}{\widetilde}

\newcommand{\Z}{\mathbb{Z}}
\newcommand{\R}{\mathbb{R}}

\renewcommand{\C}{\mathbb{C}}
\newcommand{\Cc}{\mathbb{C} \setminus 0}

\newcommand{\D}{\Delta}

\newcommand{\Ps}{\mathbf P}

\newcommand{\Df}{\mathcal{D}}

\newcommand{\Ds}{\mathbf{\Delta}}

\renewcommand{\O}{\mathbf{O}_S}

\usepackage{abstract}
\DeclareMathOperator{\rk}{rk}

\DeclareMathOperator{\Vol}{Vol}
\DeclareMathOperator{\MV}{MV}
\DeclareMathOperator{\mult}{mult}

\DeclareMathOperator{\ind}{ind}

\DeclareMathOperator{\Conv}{Conv}
\DeclareMathOperator{\Voff}{Voff}
\DeclareMathOperator{\aff}{aff}

\title{Semi-interlaced polytopes}

\author{
Fedor Selyanin \thanks{%\textit{Krichever Center for Advanced Studies}, 
Skolkovo Institute of Science and Technology, Moscow, Russia\\
\textit{Department of Mathematics}, National Research University ``Higher School of Economics'', Moscow, Russia \\
\textit{International Laboratory of Cluster Geometry}, National Research University ``Higher School of Economics'', Moscow, Russia \\
\textit{Email}: Fedor.Selyanin@skoltech.ru}
}
\date{}

\begin{document}

\maketitle

%\textcolor{red}{
%\begin{itemize}
 %   \item Пропустить через deepseek
%\end{itemize}
%}

\begin{abstract}

The Minkowski mixed volume of $n$ subpolytopes $D_1, \dots, D_n$ of a polytope $P \subset {\mathbb R}^n$ clearly does not exceed the normalized volume $n! \Vol(P)$. Equality holds if and only if the subpolytopes are \emph{interlaced}, i.e., each proper face $F \subsetneq P$ intersects at least \mbox{$\dim(F) + 1$} of the polytopes $D_i$. Efficiently computing mixed volumes for more general collections of subpolytopes is crucial for estimating the complexity of numerically solving polynomial systems.

Motivated by relaxing the bound $\dim(F) + 1$ to $\dim(F)$, we prove a combinatorial formula for the mixed volume of a broad class of \emph{semi-interlaced} polytopes. This class includes, in particular, the \emph{off-coordinate polytopes} used in computing algebraic degrees -- such as Maximum Likelihood, Euclidean Distance, and Polar degrees -- via the Kouchnirenko--Bernshtein theory. We also present applications of our results to the \emph{Arnold monotonicity problem} \cite[1982-16]{Arn04}, which concerns the dependence of Milnor numbers on the Newton polyhedra.

\end{abstract}



\setcounter{tocdepth}{2}
	\tableofcontents

\section{Introduction}

The mixed volume is the polarization of the volume form; thus, the mixed volume $\MV(P,\dots,P)$ of $n$ copies of the polytope $P \subset \R^n$ equals the normalized volume $n! \Vol(P)$, see \S\ref{Convex-sub-sec}. The mixed volume is monotonic in its arguments, meaning that for any collection of subpolytopes $D_1,\dots ,D_n \subset P$, the inequality $\MV(D_1,\dots, D_n)\le n! \Vol(P)$ holds. A natural question arises: when does equality occur?

The answer is as follows. Assume $\dim(P) = n$, since otherwise the equality holds trivially. Then $\MV(D_1,\dots,D_n) = n! \Vol(P)$ if and only if the polytopes $D_1,\dots,D_n$ are \textbf{interlaced} in $P$, i.e., for any proper face $F \subsetneq P$, at least $\dim(F) + 1$ of the polytopes $D_i$ intersect $F$. A more general ``mixed'' version of this statement also exists, where a tuple of polytopes $P_1,\dots,P_n$ is considered instead of a single polytope $P$ (with $D_i\subset P_i$). The complete answer for rational polytopes in the mixed case was first given in \cite[\S 4.7, Corollary 9]{Roj94} and later extended to all convex polytopes in \cite{BS19} (see also \cite[Remark 3.6]{BS19}). The conditions for the “unmixed” case discussed here were first presented in \cite[Remark 3.9]{Es05} and later rediscovered in \cite[Corollary 3.7]{BS19}.

This problem was initially considered in \cite{Roj94} in the context of theoretical computer science. Recall that by the Bernstein–Kouchnirenko theorem (also known as the BKK bound), the number of solutions of a (sparse) polynomial system equals the mixed volume of the corresponding Newton polytopes in the Newton nondegenerate case (see \S \ref{bk_th_sec}). A standard method for numerically solving a polynomial system $f$ is the homotopy continuation method (see, e.g., \cite{CL}). Briefly, the method works as follows: first, we consider a start system $g$ with the same Newton polytopes, whose solutions are known; then, we construct a homotopy from 
$f$ to $g$ and track the solutions along the homotopy (using a standard ODE predictor method). The mixed volume equals the number of solutions and thus serves as a key measure of the method's computational complexity.

%Some sufficient (but not necessary) conditions are also given in \cite{Che19}. These conditions are important for computational geometry (see \cite{Che19}). Calculation of the usual volume is much more efficient than calculation of the mixed volume (see \cite{BEF00} and \cite{CL15}) so these conditions provide efficient calculation of the mixed volume of certain collections of polytopes.

We now consider the following question: how does the mixed volume change if we weaken the bound from $\dim(F)+1$ to $\dim(F)$? We partially answer this question by defining \textbf{semi-interlaced} polytopes and providing a combinatorial formula for their mixed volume in terms of the usual volumes of certain associated polytopes (see Theorem \ref{semith}). The answer is partial because we impose additional restrictions on the polytopes (see Definition \ref{daughter}). Strictly speaking, semi-interlaced polytopes do not generalize interlaced polytopes (see Figure \ref{tetra}). We present two different proofs of the formula for the mixed volume of semi-interlaced polytopes:

\begin{itemize}
    \item \S \ref{convgeom_pr_sec}: A convex-geometric proof using a formula for the mixed volume;
    \item \S \ref{algem_pr_sec}: An algebro-geometric proof based on the geometry of projective toric varieties.
\end{itemize} 

The first proof is purely combinatorial, relying on a formula of Khovanskii for the mixed volume (see \S \ref{Kh-formula_subsec}), and applies to arbitrary convex polytopes. The second proof interprets the formula in terms of singularities of projective toric varieties (recalled in \S\ref{pr_tor_subsec} and \ref{mult_tor_var_subsec}) but is restricted to lattice polytopes. Note that not every convex polytope is combinatorially equivalent to a convex lattice polytope (see, e.g., \cite{Z08}).

Computing the mixed volume is also an important problem, but is a significantly harder numerical problem than computing the usual volume (see, e.g., \cite{DGH98} and \cite{Che19}). The formula for the mixed volume of semi-interlaced polytopes is computationally efficient because it expresses the mixed volume in terms of volumes of lower-dimensional polytopes associated with the given polytope, without requiring Minkowski summation.

A concrete example of semi-interlaced polytopes that arises in applications is given below. A more general version of this definition is provided in \S \ref{off_sec}.

\begin{definition}[cf. Definition \ref{off_glob_def}] \label{off_def}
    Consider a finite set $\Ps\subset \Z^m_{\ge 0} \oplus \Z^{n-m}$ with coordinates $x_1,\dots, x_m, x_{m+1}, \dots, x_n$. The \textbf{off-coordinate} polytopes are defined as $$D_i = \begin{cases}
        \Conv(\Ps \setminus \{x_i = 0\}) &  1\le i \le m, \\
        \Conv(\Ps) & m < i \le n.
    \end{cases}$$
\end{definition}

In \S \ref{alg_deg_subsec} we show that the mixed volume of off-coordinate polytopes appears in the computation of algebraic degrees (Maximum Likelihood, Euclidean Distance and Polar degrees). Motivated by the connections to Arnold’s monotonicity problem (\cite[1982-16]{Arn04}; see, e.g., \cite{Sel24} and \cite{Sel25} for a more up-to-date version), we discuss when the mixed volume of off-coordinate polytopes vanishes (see Questions \ref{neg_comb_que} and \ref{voff_2=0_que}). 

%Also note that such formulas for the mixed volume are important for the efficient calculation of the mixed volume (see e.g. \cite{Roj94} and \cite{Che19}), for instance, in the context of the homotopy continuation method for solving polynomial systems (see e.g. \cite{CL}).

\section{Preliminaries}

\subsection{Polyhedral geometry and the mixed volume}\label{Convex-sub-sec}

This subsection recalls basic notions of mixed volumes and support functions. For a comprehensive introduction to convex geometry, we refer to \cite{Ew}. Throughout this paper, we consider only convex polytopes.

\begin{definition} \label{polytope_def}
A \textbf{polytope} is the convex hull of a finite set of points in $\R^n$. A polytope is called \textbf{lattice} if all its vertices belong to the lattice $\Z^n$.
\end{definition}

\begin{definition} \label{polyhedron_def}
A \textbf{polyhedron} is the intersection of finitely many half-spaces. A bounded polyhedron is a polytope.
\end{definition}

\begin{definition}
The \textbf{Minkowski sum} $P_1 + P_2$ of two polytopes is the polytope defined as $\{p_1+p_2,  \mid p_1 \in P_1, \ p_2 \in P_2\}$. If we consider $\R^n$ as an affine space, then the Minkowski sum is defined up to translation.
\end{definition}

\begin{sign}\label{lattice_vol_sign}
Let $A$ be an affine space equipped with a full-rank lattice $L \subset A$. We use the \textbf{lattice volume} $\Vol_\Z$, which is normalized so that the volume of a unit simplex of $L$ equals $1$. In particular, the lattice volume of any lattice polytope is an integer. In the standard space $\R^n$, we consider the standard lattice $\Z^n$. 
%When performing operations on affine spaces, we assume the corresponding lattices transform accordingly.
\end{sign}

\begin{sign}\label{volumes_in_lattices_sign}
Consider an affine space $A$ with a full-rank lattice $L$. An affine subspace $A' \subset A$ is called \textbf{rational}, if $\rk (\aff (A') \cap L) = \dim (\aff (A'))$. In this case, we assume that $A'$ is equipped with the lattice $L' = L \cap A'$, and that the quotient space $A / A'$ is equipped with the quotient lattice $L / L'$.
\end{sign}

\begin{sign}\label{projection_sign}
Suppose that $F$ is a subset of an affine space $A$ with a full-rank lattice $L \subset A$, and that its affine span $\text{aff} (F)$ is rational (with the corresponding lattice $K \subset \text{aff} (F)$). Denote by $$\pi_F: A \to A/\text{aff}(F)$$ the map that sends  $\text{aff} (F)$ to the origin and whose restriction on the corresponding lattices, $$\pi_F|_L: L \to L/ K,$$ is surjective.
\end{sign}

\begin{definition}\label{mv}
The \textbf{mixed volume} $\MV$ is the unique symmetric multilinear function on $n$-tuples of polytopes in an affine space $\R^n$ with full-rank lattice $\Z^n\subset \R^n$ that satisfies the following conditions:
		\begin{enumerate}
			\item It is symmetric in its arguments.
			\item $\MV(P, \dots, P) = \Vol_\Z(P)$ for any polytope $P$.\\
			\item For any polytopes $P_1,P_2,\dots , P_n$ and $P_1^\prime$ and non-negative numbers $\lambda, \lambda^\prime$ the following equation holds 
			$$\MV(\lambda P_1 + \lambda^\prime P_1^\prime, P_2, \dots, P_n) = \lambda \cdot \MV(P_1,P_2, \dots, P_n) + \lambda^\prime \cdot  \MV(P_1^\prime, P_2, \dots, P_n),$$
			where the sum in the left hand side is the Minkowski sum and multiplication is the homothety.
		\end{enumerate}
\end{definition}

Note that the mixed volume of lattice polytopes is a non-negative integer.

\begin{definition}\label{support_face_def}
Let $P \subset \R^n$ be a polytope (or polyhedron) and $\xi \in (\R^n)^*$ a covector. The \textbf{support face} $P^\xi$ is the face of $P$ where the linear functional $\xi$ attains its maximum. The \textbf{support function} $H(\xi;P)$ is this maximum value. For a finite set $\Ps$, a \textbf{face} is the intersection of $\Ps$ with a face of $\Conv(\Ps)$. The support function for finite sets is defined analogously. For a finite lattice set $\Ps$, we define $\Vol_\Z(\Ps) = \Vol_\Z(\Conv(\Ps))$ and $\dim(\Ps) = \dim(\Conv(\Ps))$.
\end{definition}

\begin{proposition}(\cite[Assertion on p.41]{Kh78}; see \cite[Lemma 1.2]{Es08} for a proof) \label{mv=0_prop}
The mixed volume of polytopes is zero if and only if some $k$ of them have a Minkowski sum of dimension strictly less than $k$.
\end{proposition}

\begin{proposition}\label{mv=prod_rem}(see, e.g., \cite[Lemma 4]{ST10} for a proof)
Let $P_1,\dots,P_n \subset \R^n$ be polytopes with $P_1,\dots,P_k$ contained in a $k$-dimensional affine subspace $Q$. Then:
$$\MV(P_1,\dots,P_n) = \MV(P_1,\dots,P_k) \cdot \MV(\pi_Q(P_{k+1}) ,\dots, \pi_Q(P_n)),$$
where the mixed volume $\MV(P_1,\dots,P_k)$ is computed in $Q$.
\end{proposition}

\subsection{Bernstein--Kouchnirenko theorem (BKK bound)}  \label{bk_th_sec}

This subsection recalls the Bernstein–Kouchnirenko theorem  from \cite{Ber75} (also known as the BKK bound; see Remark~\ref{BKK_bound_rem}). For a detailed survey of the Bernstein–Kouchnirenko–Khovanskii theory, we refer to \cite{EKK}.

The \textbf{support} of a Laurent polynomial is the set of exponents of its nonzero monomials. The \textbf{Newton polytope} is the convex hull of the support.

\begin{definition}\label{restriction_def}
Let $f = \sum_{\omega \in \Z^n} a_\omega z^\omega$ be a Laurent polynomial with Newton polytope $P \subset \R^n$. For a polytope (or a set) $\Gamma \subset \R^n$ the \textbf{restriction} $f|_\Gamma$ is defined as $\sum\limits_{\omega \in \Gamma} a_\omega z^\omega$. We denote by $f^\xi$ the restriction $f|_{P^\xi}$ on the support face.
\end{definition}

\begin{theorem}[Bernstein--Kouchnirenko \cite{Ber75}] \label{bk_th}
Let $f_1, \dots, f_n \in \C[z_1^{\pm 1}, \dots, z_n^{\pm 1}]$ be Laurent polynomials with Newton polytopes $P_1, \dots, P_n$ and nondegenerate coefficients. Then the number of solutions to $f_1 = \dots = f_n = 0$ in the torus $(\C \setminus 0)^n$ equals the mixed volume $\MV(P_1, \dots, P_n)$.

The nondegeneracy condition requires that for any nonzero covector $\xi \in (\R^n)^* \setminus 0$, the system $f_1^\xi = \dots = f_n^\xi = 0$ has no solutions in the torus $(\Cc)^n$.
\end{theorem}

The following corollary follows from Theorem \ref{bk_th} and Proposition \ref{mv=0_prop}, since we can consider each $f_i$ as a Laurent polynomial in variables $z_1,\dots, z_n, z_{n+1}$ that does not depend on $z_{n+1}$.

\begin{cor}\label{gen_sys_no_sol_rem}
Let $\Ps_1, \dots, \Ps_{n+1} \subset \Z^n$ be finite lattice sets, and let $f_1, \dots, f_{n+1} \in  \C[z_1^{\pm 1}, \dots, z_n^{\pm 1}]$ be Laurent polynomials with supports $\Ps_1, \dots, \Ps_{n+1}$ respectively and generic coefficients. Then the system $f_1 = \dots = f_{n+1} = 0$ has no solutions in $(\Cc)^n$. 
\end{cor}

\begin{Remark}\label{BKK_bound_rem}
The Bernstein–Kouchnirenko theorem is often referred to as the \textbf{BKK bound}. This name emphasizes that the number of isolated solutions in the torus $(\Cc)^n$ is bounded from above by the mixed volume $\MV(P_1, \dots, P_n)$, even if the nondegeneracy condition is violated.
%when the nondegeneracy condition is violated, the number of isolated solutions in the torus $(\mathbb{C}^*)^n$ becomes strictly less than $\MV(P_1, \dots, P_n)$.
\end{Remark}

\subsection{Interlaced polytopes} \label{inter_sec}

This subsection recalls the notion of interlaced polytopes, i.e. tuples of subpolytopes of a full-dimensional polytope whose mixed volume equals the lattice volume of the polytope.
    
	\begin{definition}\label{intdef}
		Polytopes $P_1,\dots,P_n$ are \textbf{interlaced} in $P = \Conv(\bigcup_i P_i)$ if for every face $F\subsetneq P$ of $P$ at least $\dim(F) + 1$ of the polytopes $P_1,\dots,P_n$ intersect $F$.
	\end{definition}

\begin{figure}[H]
		\begin{center}
	\includegraphics[scale=.2]{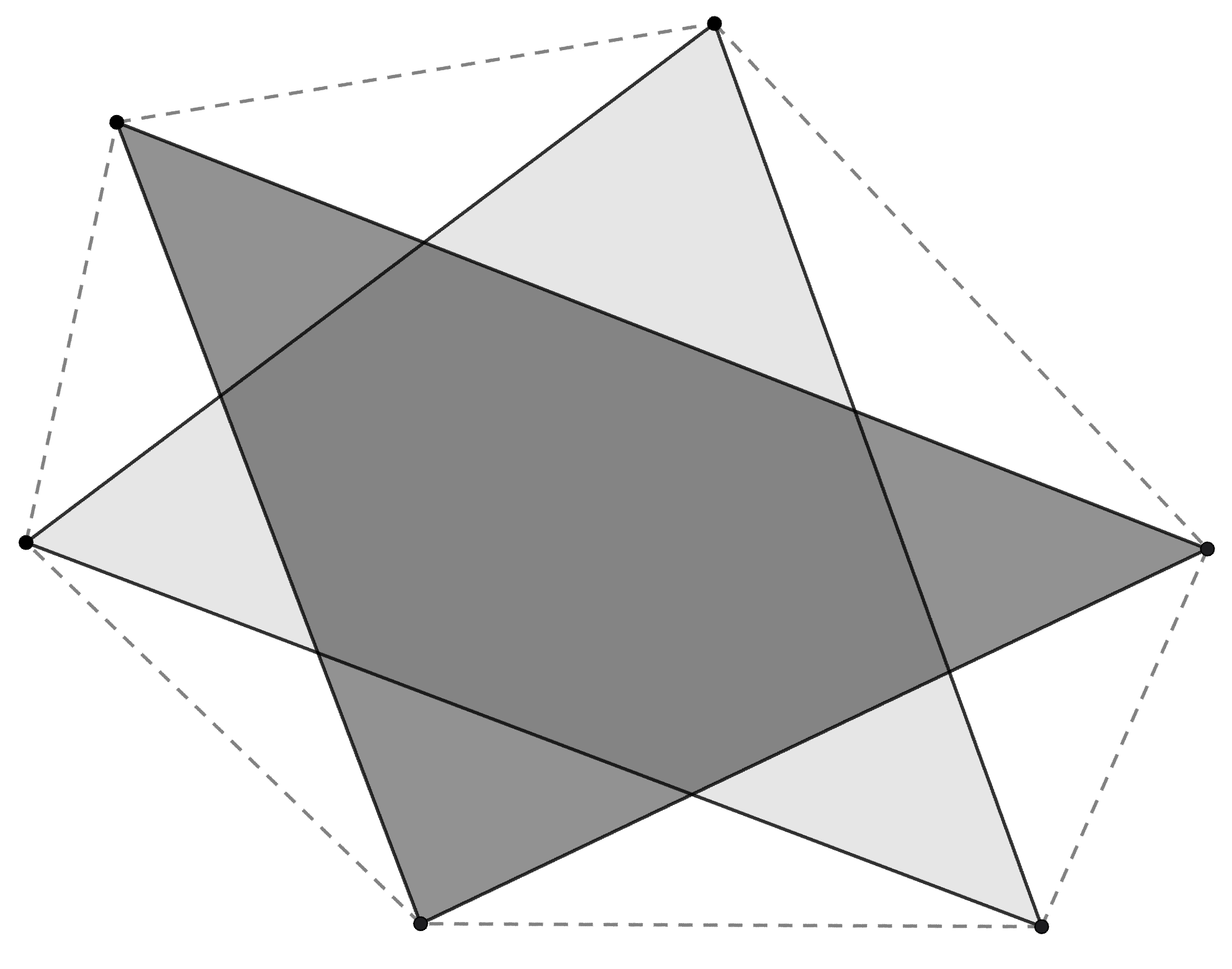}
		\end{center}
		\caption{
		\label{interlaced}	Interlaced triangles.}
\end{figure}

	\begin{proposition}(\cite[Remark 3.9]{Es05} or \cite[Corollary 3.7]{BS19})\label{interlaced_prop}
		If polytopes $P_1,\dots,P_n$ are interlaced in $P$, then their mixed volume $\MV(P_1,\dots, P_n)$ equals $\Vol_\Z(P)$. If $\dim(P) = n$, the converse also holds.
	\end{proposition}

We recall a proof of this proposition using projective toric varieties in \S {\ref{algem_pr_sec}}.
 
\subsection{A formula of Khovanskii for the mixed volume}\label{Kh-formula_subsec}

%\begin{definition}
%The \emph{support face} $P^\xi$ of the polyhedron $P$ is the maximal face of $P$ on which the minimum of the functional $\xi$ is achieved. The support function $\xi(P)$ of the polyhedron $P$ is the minimal value of $\xi$ on the polyhedron $P$. A \emph{face} of a finite set $\Ps \subset \R^n$ is the intersection of a face of its convex hull $\Conv(\Ps)$ with the set $\Ps$. The \emph{support face} and the \emph{support function} of a finite set are defined in the same way as for convex polyhedra.
%\end{definition}

Consider a covector $\xi \in (\R^n)^*$ and a polyhedron $\tl P \subset \R^n \oplus \R^1$. If the linear functional $(\xi,1) \in (\R^n)^* \oplus (\R^1)^*$ is bounded from above on $\tl P$, then we denote by $\tl P (\xi)$ the polytope $\tau (\tl P^{(\xi,1)}) \subset \R^n$, where $\tau: \R^n\oplus \R^1 \to \R^n$ is the projection onto the first summand.

Note that the following formula is not the one mentioned in the title of \cite{Kh99}.

\begin{lemma} (\cite[\S 4, Lemma 2]{Kh99}) \label{hovamv}
Let $P_1, \dots, P_n \subset \R^n$ be polytopes. Suppose $\Ps_i \subset \R^n$ are finite sets such that $\Conv(\Ps_i) = P_i$. Consider functions $\psi_i: \Ps_i \to \R^1$ and define $\tl P_i \subset \R^n \oplus \R^1$ as the convex hull of the set 
$\{(\mathbf{p}, y) \mid \mathbf{p} \in \Ps_i,  y \le \psi_i(\mathbf{p})\}$. Then:
		\begin{equation} \label{FormulaHova}
		\MV(P_1,\dots ,P_n) = \sum\limits_{\xi \in (\R^n)^*} \MV(\widetilde{P}_1(\xi), \dots, \widetilde{P}_n(\xi)).
		\end{equation}
\end{lemma}

\begin{Remark}
Although there are infinitely many covectors $\xi \in (\R^n)^*$, only finitely many terms in the right-hand side of equation (\ref{FormulaHova}) are nonzero. For instance, all nonzero terms correspond to dual rays to the facets of the Minkowski sum $\widetilde{P}_1 + \dots + \widetilde{P}_n$.
\end{Remark}

\begin{Remark}
In \cite[\S 4, Lemma 2]{Kh99}, this formula is given in a more general form. It is formulated in terms of so-called consistencies of partitions satisfying the condition of inheritance. Recall that a subdivision of a polytope $P$ given by the collection of polytopes $\{\widetilde{P}(\xi) \mid \xi \in (\mathbb{R}^n)^*\}$ for some polyhedron $\widetilde{P} \subset \mathbb{R}^n \oplus \mathbb{R}^1$ with $\tau(\widetilde{P}) = P$ is called a regular subdivision (or coherent subdivision in \cite{GKZ94}). In \cite{Kh99}, it is mentioned that regular subdivisions induce regular consistencies with the property of inheritance. This implies the version of the formula that we use in our paper. 
\end{Remark}

\begin{figure}[H]
		\begin{center}
	\includegraphics[scale=1]{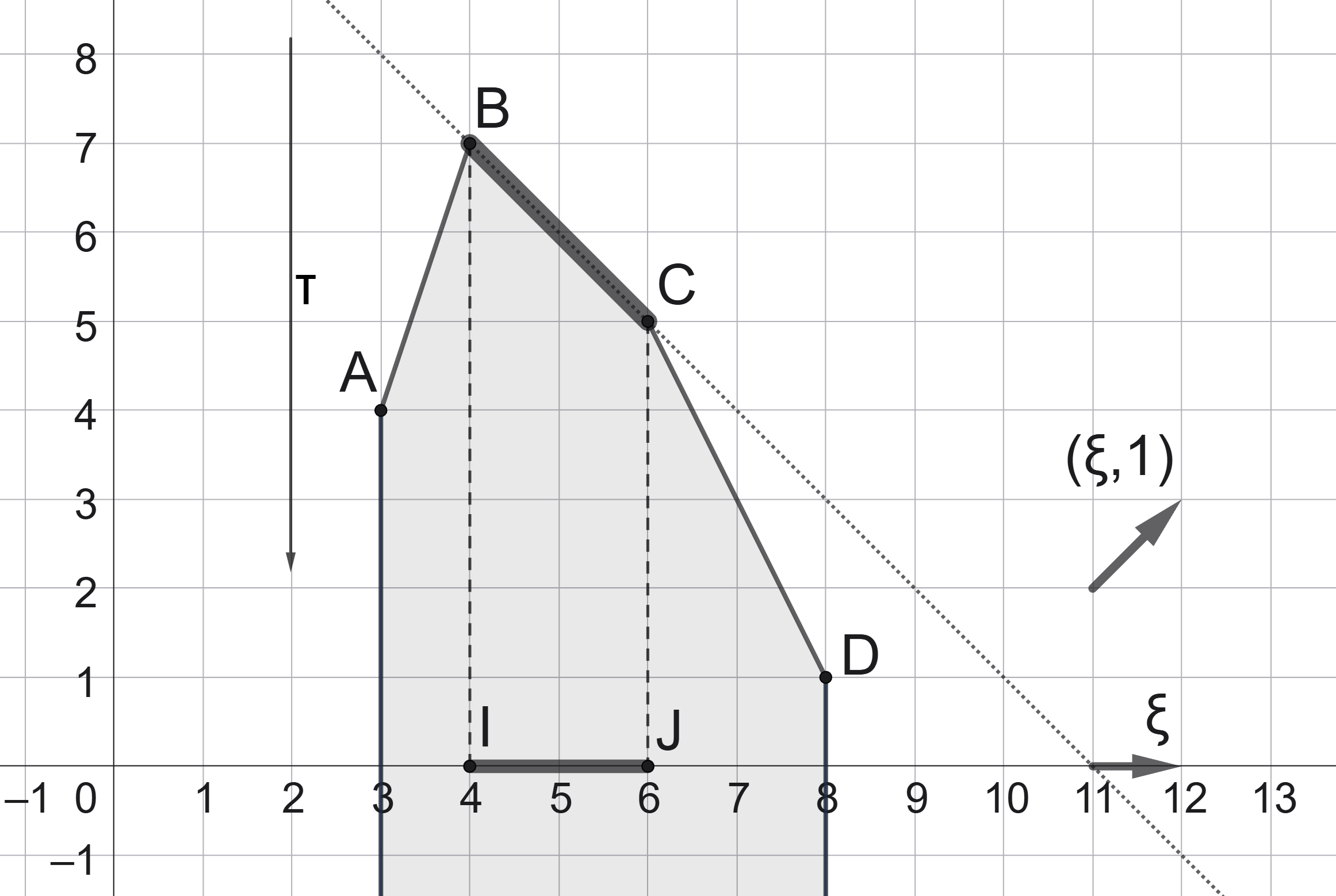}
		\end{center}
		\caption{
		\label{kh_formula_fig} Polyhedron $\tl P \subset \R^1 \oplus \R^1$ with vertices \{A,B,C,D\}, covector $\xi = (1)$, support face $BC = \tl P^{(\xi,1)}$, and segment $IJ = \tl P (\xi)$.}
\end{figure}

\subsection{Projective toric varieties} \label{pr_tor_subsec}

This subsection recalls basic facts about projective toric varieties; for a comprehensive introduction, see \cite{Ful}.

Let $\Ps = \{a_1,\dots,a_m\} \subset \Z^n$ be a finite lattice set. The corresponding \textbf{projective toric variety} $X_\Ps$ is defined as the closure of the image of the map:
$$ \phi_\Ps : (\C\setminus 0)^n \to  \mathbb P (\C^{|\Ps|}) \quad z \to (z^{a_1}:z^{a_2}:\dots :z^{a_m}).$$
The orbits of the torus action on $X_\Ps$ correspond to faces of $\Ps$ and have the same complex dimension as the corresponding face. 

\begin{sign}\label{p-p_lattice_sign}
Denote by $\langle \Ps - \Ps\rangle$ the lattice span by the set $\Ps - \Ps \stackrel{\mathclap{\normalfont\mbox{\normalfont\tiny def}}}{=} \{p_1-p_2 \mid  p_1, p_2 \in \Ps \}$.
\end{sign}

\begin{Remark} \label{orbits_lattices_rem}
Specifically, the orbit $T_{\Ps'}$ corresponding to a face $\Ps'$ of $\Ps$ is naturally isomorphic to the torus with monomial lattice $\langle\Ps' - \Ps'\rangle$ and with character lattice $(\Z^n)^* / \aff (\Ps')^\bot$, where $\aff (\Ps')^\bot$ consists of all covectors constant on $\Ps'$.
\end{Remark}

\begin{Remark}\label{f=0_is_hyperplane_rem}
Consider a finite set $\Ps \subset \Z^n$ and a Laurent polynomial $f = \sum_{\omega_i \in \Ps} a_i z^{\omega_i}$. The image $\phi_\Ps (\{f = 0\}) \subset \mathbb P (\C^{|\Ps|})$ of the hypersurface $\{f = 0\} \subset (\Cc)^n$ is the section of the open orbit $\phi_\Ps( (\Cc)^n)$ of the toric variety $X_\Ps \subset \mathbb P (\C^{|\Ps|})$ with the hyperplane $\sum_{\omega_i \in \Ps} a_i \alpha_i$, where $\alpha_i$ is the homogeneous coordinate of $\mathbb P (\C^{|\Ps|})$ corresponding to the monomial $z^{\omega_i}$.
\end{Remark}

\begin{definition}\label{index_def}
Let $\Ps \subset L$ be a finite lattice set in a lattice $L$, and suppose that $\aff(\Ps)$ contains $L$. The \textbf{index} $\ind_L(\Ps)$ is defined as the order $|L / \langle \Ps - \Ps\rangle|$ of the finite abelian quotient group $L / \langle \Ps - \Ps\rangle$.
\end{definition}

\begin{Remark}\label{isom_toric_rem}
Let $\Ps$ be a finite set in a lattice $L$. The projective toric variety $X_\Ps$ is independent of the choice of $L$ containing $\Ps$ and the choice of origin in $L$ (i.e., it is invariant under parallel translations of $\Ps$). In particular, every projective toric variety $X_\Ps$ is isomorphic to one defined by a \textbf{spanning set} (i.e., a set $\Ps \subset L$ such that $\langle \Ps - \Ps\rangle = L$).
\end{Remark}

\begin{Remark}\label{BK_toric_rem}
If the support of a polynomial $f$ is contained in $\Ps$, then $f$ defines a hyperplane section of $X_\Ps \subset \mathbb{P}(\C^{|\Ps|})$. The solutions of a system of $n$ linearly independent equations with supports in $\Ps$:
$$f_1 = \ldots = f_n = 0$$ 
in $(\Cc)^n$ correspond to the intersection of the open orbit of $X_\Ps$ with a projective subspace  $H \subset \mathbb{P}(\C^{|\Ps|})$ of codimension $n$ (see Remark \ref{f=0_is_hyperplane_rem}). The Kouchnirenko--Bernstein theorem (together with Remark \ref{isom_toric_rem}) implies that the number of intersection points $H \cap X_\Ps$ counted with multiplicities (i.e., the degree of $X_\Ps \subset \mathbb P(\C^m)$), equals $\frac{\Vol_\Z(P)} {\ind_{\Z^n} (\Ps)}$. The nondegeneracy conditions in the Bernstein–Kouchnirenko theorem are equivalent to the requirement that no solutions lie in toric orbits of positive codimension.
\end{Remark}

\subsection{Multiplicities of orbits of projective toric varieties} \label{mult_tor_var_subsec}

This subsection recalls, following \cite[\S 5.3]{GKZ94}, the multiplicities of projective toric varieties along their orbits, and, from \cite[\S 1.6]{Es08}, the corresponding nondegeneracy conditions for sections by projective subspaces of complementary dimension with respect to the orbits.

Let $F \subset \R^n$ be a set whose affine span $\aff(F)$ is a rational subspace (see Notation \ref{volumes_in_lattices_sign}). %Denote by $\pi_F$ the linear lattice projection $\R^n \to \R^n / \aff(F)$ with kernel $\aff(F)$.

\begin{definition}\label{c_def}
For lattice sets $\Fs, \Fs^\prime \subset \Z^n$, define the combinatorial coefficients $c_\Fs^{\Fs^\prime}$ as follows:
\begin{enumerate}
        \item $c_\Fs^{\Fs^\prime} = 0$, if $\Fs^\prime$ is not a face of $\Fs$.
        \item $c_\Fs^{\Fs^\prime} = \Vol_\Z(\pi_{\Fs'} (\Fs)) - \Vol_\Z(\pi_{\Fs'} (\Fs \setminus \Fs'))$, if $\Fs^\prime$ is a face of $\Fs$.
        \item $c_\Fs^{\Fs} = 1$  (the $0$-dimension volume of a point is $1$).
    \end{enumerate}
\end{definition}

See \cite[Examples 1.24 and 1.26]{Es08} for illustrations of the latter definition.

\begin{definition}\label{mult_def}
Let $Y \subset \mathbb{P}^m$ be a projective variety and $y \in Y$. Consider a generic projective subspace $H$ passing through $y$, such that $\dim (Y) + \dim (H) = m$. The number of intersection points in $H(t) \cap Y$ arising from the isolated intersection $y \in H \cap Y$ under a generic deformation $H(t)$ of $H$ is called the \textbf{multiplicity of $Y$ at $y$} and denoted by $\mult_y Y$. Let $Z$ be an irreducible subvariety of $Y$. The multiplicity $\mult_z Y$ for a generic point $z \in Z$ is called the \textbf{multiplicity of} $Y$ \textbf{along} $Z$ and is denoted by $\mult_Z Y$.
\end{definition}

\begin{Remark}
The multiplicity $\mult_y Y$ is always at least $1$, and $\mult_y Y = 1$ if and only if $y$ is a nonsingular point of $Y$.
\end{Remark}

\begin{theorem}(\cite[Theorem 3.16 on p.~187]{GKZ94} and Remark \ref{isom_toric_rem}) \label{toric_orb_mult_th}
Let $\Ps$ be a finite set in a lattice $L$ with $\dim(\Ps) = \rk(L)$, and let $\Ps'$ be a face of $\Ps$. Denote by $L'$ the lattice $L \cap \aff (\Ps')$.
The multiplicity of the projective toric variety $X_\Ps$ along the orbit $X_{\Ps'}$ is given by:
    $$\mult_{X_{\Ps'}} (X_{\Ps}) = \frac{\ind_{L'} \Ps'}{\ind_L \Ps} c_\Ps^{\Ps'}.$$
\end{theorem}

\begin{Remark}\label{smooth}
     The toric variety $X_\Ps$ is smooth if and only if all multiplicities $\frac{\ind_{L'} \Ps'}{\ind_L \Ps} c_\Ps^{\Ps'}$ are equal to $1$. Equivalently, in the lattice affinely generated by $\Ps$, all coefficients $c_\Ps^{\Ps'}$ and indices $\ind_{L'} \Ps'$ are equal to $1$.
\end{Remark}

Using \cite[Lemma 1.28 and Definition 1.14]{Es08}, we can state explicit sufficient nondegeneracy conditions for the statement of Theorem \ref{toric_orb_mult_th} (see Lemma \ref{non_deg_cond_rem}). 

Let $\Ps'$ be a face of a finite set $\Ps \subset \Z^n$. Recall that, by Remark \ref{orbits_lattices_rem}, the lattice $\langle\Ps' - \Ps'\rangle$ from Notation \ref{p-p_lattice_sign} is the monomial lattice of the orbit $T_{\Ps'}$ of the toric variety $X_\Ps$.

%Recall that the intersection of the hyperplane $\sum\limits_{i=1}^m \alpha_i x_i = 0$ (with $\alpha_i \in (\C \setminus 0)$) with $X_\Ps$ is the closure of $\phi_\Ps (\{f = 0\})$. The intersection of the closure of $\{f = 0\}$ with the orbit $T_{\Ps'}$ corresponding to a face $\Ps'$ is given by the restriction $f_{\Ps'}$.

\begin{definition}\label{rest_face_def}
Let $f$ be a Laurent polynomial with support $\Ps\subset \Z^n$. Consider a face $\Ps' \subset \Ps$ and a point $\mathbf p' \in \Ps'$. Define the hypersurface $\{f|_{\downarrow \Ps'} = 0\} \subset T_{\Ps'}$ as $\{z ^{-\mathbf p' }f_{\Ps'} = 0\}$.
%The restriction of $\{f_{\Ps'} = 0\}$ to the orbit $T_{\Ps'}$ is denoted $\{f|_{\downarrow \Ps'} = 0\}$. This is well-defined, since shifting by an element of the subtorus with character lattice $\aff(\Ps')^*$ multiplies $f_{\Ps'}$ by a nonzero constant.
\end{definition}

\begin{Remark}
The hypersurface in the above definition is well-defined, since a different choice of $\mathbf p'$ multiplies the polynomial $z ^{-\mathbf p' }f_{\Ps'}$ by a monomial in the monomial lattice of $T_{\Ps'}$. 

If $f_{\Ps'} \not\equiv 0$, then $\{f|_{\downarrow \Ps'} = 0\} \subset T_{\Ps'}$ coincides with the intersection of the closure of $\phi_\Ps (\{f = 0\})$ with the corresponding orbit $T_{\Ps'}$.
\end{Remark}

\begin{definition}\label{substitution_def}
Let $f$ be a polynomial with support $\Ps \subset \Z^n$, and let $T \subset (\Cc)^n$ be a subtorus with character lattice $L \subset (\Z^n)^*$. For $t \in (\Cc)^n / T$, the \textbf{substitution} $f(t,\cdot)$ is the restriction of $f$ to $tT$. Note that $f(t,\cdot)$ can be considered as a Laurent polynomial with support contained in $(\Ps / L^\bot) \subset (\Z^n / L^\bot)$.
\end{definition}

\begin{definition}\cite[Definition 1.14]{Es08} \label{non_deg_def}
Let $\Ps \subset \Z^n$, $P = \Conv(\Ps)$, and let $\mathbf{p}$ be a vertex of $\Ps$. Let $f_1,\dots, f_n$ be Laurent polynomials with supports in $\Ps \setminus \mathbf p$, and let $P_{-\mathbf p}$ be the polytope $\Conv(\Ps \setminus \mathbf p)$. The system 
$$f_1 = \ldots = f_n = 0$$
is said to be \textbf{nondegenerate at} $\mathbf p$ if for every $\xi \in (\R^n)^*: P^\xi = \mathbf p$, the system 
$$f_1|_{P^\xi_{-\mathbf p}} = \ldots = f_n|_{P^\xi_{-\mathbf p}} = 0$$
has no solutions in the torus $(\Cc)^n$.
\end{definition}

\begin{lemma}\cite[Lemma 1.28]{Es08} \label{non_deg_cond_rem}
Let $\Ps \subset \Z^n$ with $\dim \Ps = n$, and let $\Ps'$ be a face of $\Ps$ with $\dim \Ps' = k$. Consider the toric variety $X_\Ps$ and its orbit $T_{\Ps'}$. Let
\begin{equation}\label{non_deg_check_eqq}
    f_1 = \ldots = f_{k} = f_{k+1} = \ldots = f_{n} = 0
\end{equation}
be a system of linearly independent equations such that the supports of $f_1, \dots, f_k$ are contained in $\Ps$, and the supports of $f_{k+1}, \dots, f_n$ are contained in $\Ps \setminus \Ps'$. %Recall that the $k$-dimensional torus $T_{\Ps'}$ in $X_{\Ps'}$ has character lattice $(\Z^n)^* / \aff (\Ps')^\bot$.
The codimension $n$ projective subspace in $\mathbb P (\C^{|\Ps|})$ corresponding to the system (\ref{non_deg_check_eqq}) intersects $X_{\Ps'} \subset \mathbb P (\C^{|\Ps|})$ generically at $z^{(k)} \in T_{\Ps'}$ (in the sense of Definition \ref{mult_def}) if:
\begin{enumerate}
    \item The system $f_1|_{\downarrow\Ps'} = \ldots = f_{k}|_{\downarrow\Ps'} = 0$ has a nondegenerate solution $z^{(k)}$ in $T_{\Ps'}$.
    \item The system $f_{k+1} (z^{(k)}, \cdot) = \ldots = f_{n} (z^{(k)}, \cdot) = 0$ is nondegenerate at the vertex $\pi_{\Ps'} (\Ps')$.
\end{enumerate}    
\end{lemma}

\section{Semi-interlaced polytopes}\label{Semi-interlaced polytopes}

%\textcolor{red}{The main ingredients of the formula for the mixed volume of semi-interlaced polytopes are the combinatorial coefficients $c_S^{S^\prime}$ (see Definition \ref{c_def}) over some faces $S,S^\prime$ of $\Ps$ called \textit{sutures}. These coefficients (up to the index of certain sublattice) are equal to the multiplicity of the projective toric variety $X_{\Ps \cap S}$ along the stratum $X_{\Ps \cap S^\prime}$ (see \cite[\S 3F - Theorem 3.16]{GKZ94}). The same coefficients provide the relation between the secondary polytope and the Newton polytope of the $\As$-discriminant (\cite[\S 10B -- Theorem 1.2]{GKZ94}). The Euler obstructions of projective toric varieties are also described in terms of these coefficients (see \cite{MT08} and \cite{Es08}). We provide the formula with two proofs}

This section defines semi-interlaced polytopes and presents a formula for their mixed volume in Theorem \ref{semith}. The main result of this paper, Theorem \ref{semith}, is proved via two different approaches in \S \ref{proofs_sec}.

\begin{definition}\label{daughter}
Let $\Ps \subset \Z^n$ be a finite set with convex hull $P = \Conv(\Ps)$. Let $D$ be a non-empty polytope with vertices in $\Ps \subset \Z^n$. Denote by $\Df = \Df(D)$ the set of all maximal (with respect to inclusion) faces $F$ of $P$ such that $D \cap F = \emptyset$. The polytope $D$ is called a \textbf{daughter polytope} of $\Ps$ if:
\begin{enumerate}
\item Any two distinct faces in $\Df$ are disjoint.
\item $D = \Conv\left(\Ps \setminus \bigcup_{F \in \Df} F\right)$.
\end{enumerate}
        	
\begin{figure}[H]
		\begin{center}
	\includegraphics[scale=.5]{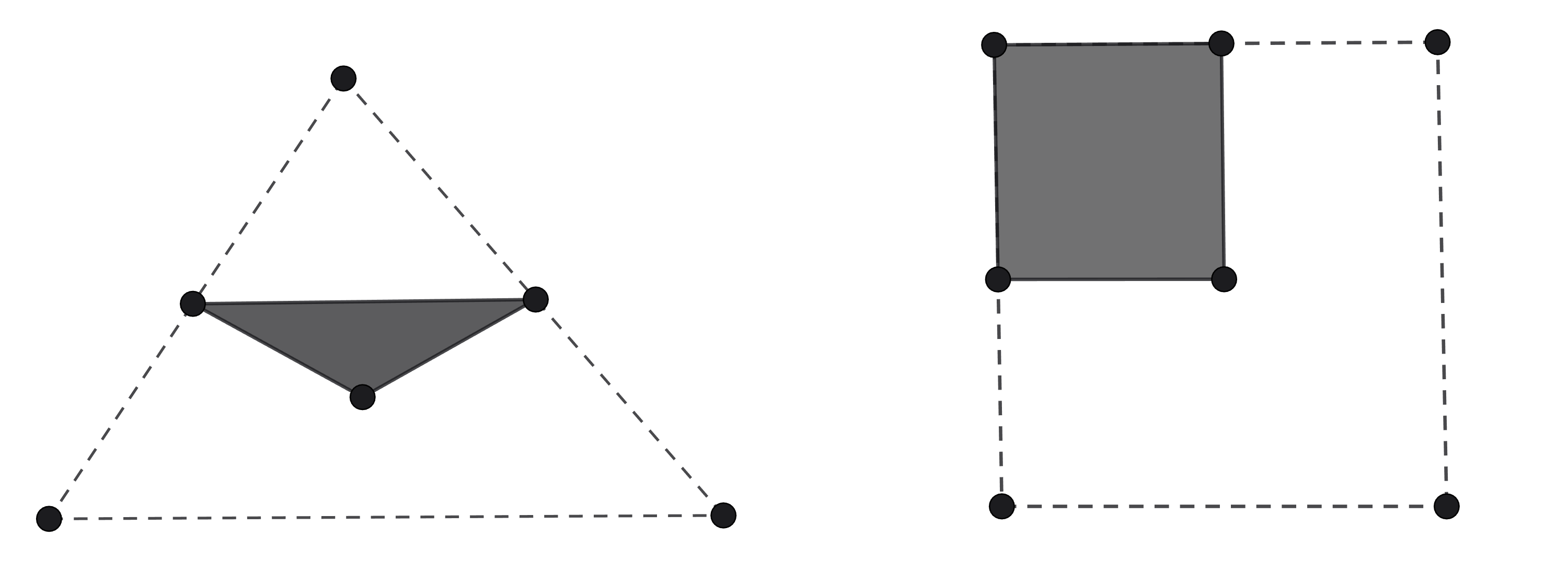}
		\end{center}
		\caption{
		\label{DaughterAndNot} Daughter (left) and non-daughter (right) polytopes.}
\end{figure}
\end{definition}

\begin{Remark}\label{face_rest_daughter_rem}
If $D$ is a daughter polytope of $\Ps$ and $F \subset P$ is a face with $F \notin \Df(D)$, then $D \cap F$ is a daughter polytope of $\Fs = F \cap \Ps$, where the corresponding set $\Df_F (D \cap F)$ consists of all non-empty intersections $R \cap F$ for $R \in \Df (D)$.
\end{Remark}
    
\begin{definition}\label{semi}
Let $\Ps \subset \Z^n$ be a finite set with convex hull $P = \Conv(\Ps)$. Daughter polytopes $D_1, \dots, D_n$ of $\Ps$ are called \textbf{semi-interlaced} in $\Ps$ if for every face $F$ of $P$, at least $\dim(F)$ of the polytopes $D_1, \dots, D_n$ intersect $F$. A face $F$ is called a \textbf{suture} of $P$ if exactly $\dim(F)$ of the polytopes intersect it. We denote the set of sutures by $\Theta(P)$.
\end{definition}

Note that $P$ itself is always a suture, but $P$ need not equal $\Conv(\bigcup_i D_i)$, as some vertices of $P$ may be sutures. Remark \ref{face_rest_daughter_rem} implies the following:
    
\begin{Remark}\label{semi_ind_rem}
For a suture $S \in \Theta (P)$, the polytopes $\{D_i\cap S| D_i\cap S \ne \emptyset\}$ are semi-interlaced in $(\Ps \cap S)$.
\end{Remark}

We do not need the following proposition for the main results of the paper, but we believe it provides some intuition concerning the sutures of semi-interlaced polytopes.

\begin{proposition}
Let $S_1, S_2 \in \Theta (P)$ be sutures such that $S_1 \cap S_2 \ne \emptyset$ and $\aff(S_1 \cap S_2) = \aff(S_1) \cap \aff(S_2)$. Then $S_1 \cap S_2$ and $P \cap \aff(S_1 \cup S_2)$ are also sutures.
\end{proposition}

\begin{proof}
Suppose $\dim (S_1 \cap S_2) = a$, $\dim (S_1) = a+b$, and $\dim (S_2) = a + c$. Then $\dim (P \cap \aff(S_1 \cup S_2)) = a + b + c$. Denote by $A$, $B$, $C$ and $BC$ the subsets of the semi-interlaced polytopes consisting of polytopes that have non-empty intersections with $S_1 \cap S_2$, $S_1$, $S_2$, and $(P \cap \aff(S_1 \cup S_2))$, respectively. Note that $A \subset B \cap C$ and $B \cup C \subset BC$. By the definition of daughter polytopes, $BC = B \cup C$, since we remove a set of disjoint faces. This implies that $|B| + |C| \ge |A| + |BC|$.

By the definition of sutures, we have $|B| = a + b$ and $|C| = a + c$. By the definition of semi-interlaced polytopes, we have $|A| \ge a$ and $|BC| \ge a + b + c$. Thus we also obtain $|B| + |C| \le |A| + |BC|$, which implies that $|B| + |C| = |A| + |BC|$, and that both $S_1 \cap S_2$ and $(P \cap \aff(S_1 \cup S_2))$ are sutures.
\end{proof}

Not every collection of interlaced polytopes is semi-interlaced. For example, consider two copies of a tetrahedron and a segment from the apex to a non-vertex point of the base (see Figure \ref{tetra}). These polytopes are interlaced but not daughter polytopes of any set, hence not semi-interlaced.
	
\begin{figure}[H]
		\begin{center}
	\includegraphics[scale=.5]{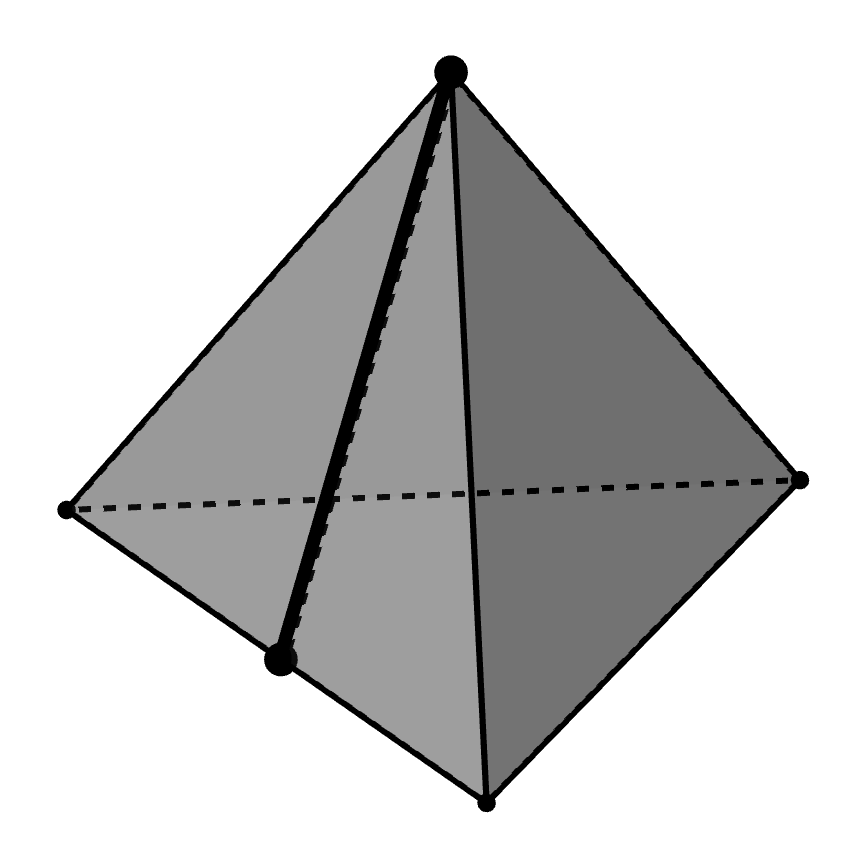}
		\end{center}
		\caption{
		\label{tetra}Interlaced but not semi-interlaced polytopes.}
\end{figure}

For a suture $S$, let $\mathfrak{v}^\dagger(S)$ denote the mixed volume $\MV(\{D_i \cap S \mid D_i \cap S \ne \emptyset\})$ (computed in $\aff(S)$), and let $\mathfrak{v}(S)$ denote the lattice volume of $S$ (see Notation \ref{lattice_vol_sign}). Denote by $\mathfrak{v}^\dagger$ and $\mathfrak{v}$ the vectors of mixed volumes and volumes, respectively, indexed by sutures.

Consider the square matrix with entries $c_{S\cap \Ps}^{S^\prime \cap \Ps}$ (see Definition \ref{c_def}), where $S$ and $S^\prime$ run over the set of sutures $\Theta(P)$. Note that this matrix is upper triangular with $1$'s on the diagonal when sutures are ordered by dimension; in particular, its determinant is $1$, and its inverse has integer entries. Denote the inverse matrix by $\mathfrak D$.
 
\begin{theorem}[Main theorem] \label{semith}
The mixed volumes of semi-interlaced polytopes satisfy:
\begin{equation} \label{semieq}
    \mathfrak v^\dagger = \mathfrak D \cdot \mathfrak v.
\end{equation}
In particular, the mixed volume $\MV(D_1, \dots, D_n)$ is the component $\mathfrak{v}^\dagger(P)$ corresponding to $P$.
\end{theorem}

\begin{Remark}\label{conv_geon_non_lattice_rem}
Although we work with lattice polytopes, the convex-geometric proof in \S \ref{convgeom_pr_sec} extends to non-rational polytopes without change. Note that not every convex polytope is combinatorially equivalent to a lattice polytope (see, e.g., \cite{Z08}).
\end{Remark}

\begin{Remark}\label{=Vol}
If $\Theta (P) = \{P\}$, then $\mathfrak{D}$ is the $1 \times 1$ identity matrix, and the daughter polytopes are interlaced and their mixed volume equals $\Vol_\Z(P)$, which is consistent with (\ref{semieq}).
\end{Remark}

\begin{Remark}\label{0_cci_rem}
The entries of $\mathfrak D$ are Euler obstructions of toric varieties (see \cite{MT08} for the original formulas and \cite[Section 1.5]{Es08} for our notations). A similar formula to (\ref{semieq}) computes the number of points in zero-dimensional critical complete intersections (also known as the $e$-Newton number); see \cite[Lemma 2.50]{Es18}, \cite[Theorem 1.8]{Es24} and \cite[Definition 5.20]{Sel25}.
\end{Remark}

\begin{Remark}\label{+-1_rem}
Let $S$ be a suture of $P$. Suppose the cone $\{\lambda x\mid x \in \pi_S(P), \lambda \in \R_{\ge 0}\}$ generated by $\pi_S(P)$ is simplicial, every face $S \subset F \subset P$ is a suture, and every coefficient $c_\Fs^{\mathbf S}$ equals $1$. Then the entry of $\mathfrak{D}$ corresponding to $(S, F)$ is $(-1)^{\dim F - \dim S}$.
\end{Remark}

\section{Proofs of the main theorem}\label{proofs_sec}

This section presents two proofs of Theorem \ref{semith}. 

The first proof employs convex geometry, relying on a formula of Khovanskii for the mixed volume (Lemma \ref{hovamv}) and the main Lemma \ref{near} (stated in \S \ref{lemma_defining_subsec} and proved in \S \ref{lem-subsec}). 

The second proof uses the algebraic geometry of projective toric varieties, based on the formula for multiplicities along orbits (Theorem \ref{toric_orb_mult_th}) and on Corollary \ref{for_alg_non_deg_cor}, which is obtained in the course of proving the main Lemma \ref{near}.

\subsection{The main lemma}\label{lemma_defining_subsec}

Let $\Ps \subset \Z^n$ be a finite set with convex hull $P = \Conv(\Ps)$. Let $D_1, \dots, D_n$ be semi-interlaced polytopes in $\Ps$, and let $\Df_1 = \Df(D_1), \dots, \Df_n = \Df(D_n)$ be the corresponding sets of faces of $P$.

For a suture $S \in \Theta (P)$, define a point $\mathbf{O}_S \in \pi_S(\Z^n)$ and a finite lattice set $\Ps_S \subset \pi_S(\Z^n)$ by:
\[
\O =  \pi_S(S) \quad , \quad \Ps_S =  \pi_S(\Ps) \cap \overline{\pi_S(P) \setminus \pi_S(\Conv(\Ps \setminus S))},
\]
where the overline denotes topological closure.

\begin{figure}[H]
\begin{center}
\includegraphics[scale=.5]{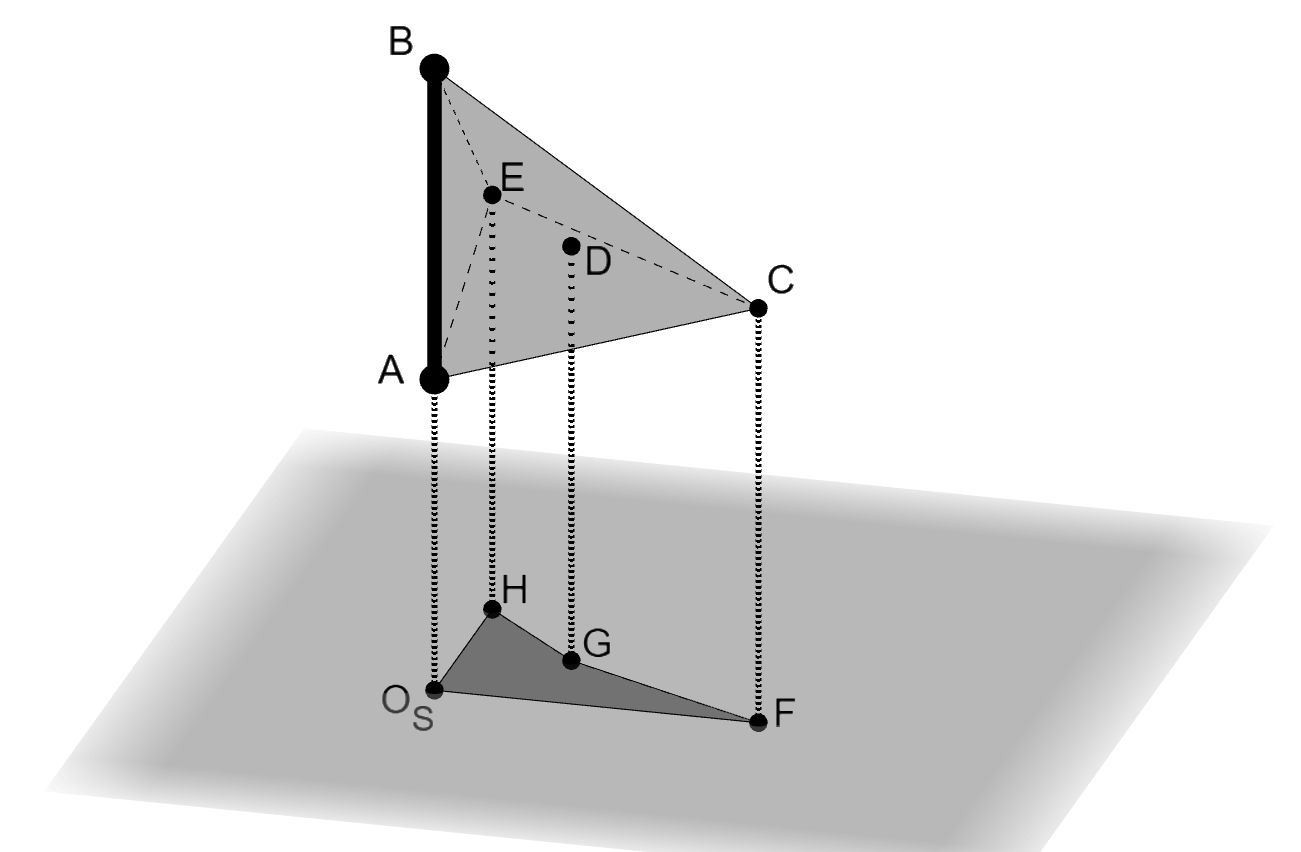}
\end{center}
\caption{
\label{O_S_sign_fig} Set $\Ps = \{A,B,C,D,E\}$, suture $AB$, non-convex polygon $\overline{\pi_S(P) \setminus \pi_S(\Conv(\Ps \setminus S))}$, and set $\Ps_S = \{\O, F,G,H\}$. 
}
\end{figure}

Define $\widetilde{P}_S \subset \pi_S(\R^n) \oplus \R^1$ as the convex hull of the following rays: 
$$\widetilde{P}_S = \Conv (\{(\mathbf a,y)\mid \ \mathbf a\in \Ps_S\setminus \O,\ y \le 0\} \cup  \{(\O,y), \ y \le -1\}).$$ 
Define $\tl D_i \subset \R^n \oplus \R^1$ as the convex hull of the following rays: 
$$\tl D_i = \Conv (\{(a,y)\mid \  a \in D_i ,\ y \le 0\} \cup  \{(a,y)\mid \  a \in P,\ y \le -1 \}).$$

The following lemma uses notation from \S \ref{Kh-formula_subsec}; its proof is given in \S \ref{lem-subsec}.

\begin{lemma}[Main lemma] \label{near}
Let $D_1, \dots, D_n$ be semi-interlaced polytopes in $\Ps$, and let $S \in \Theta (P)$ be a suture. Consider a covector $\xi \in (\pi_S(\R^n))^*$ such that $(\pi_S(P))^\xi = \mathbf O_S$.  By definition of a suture, the set $I = \{i \mid D_i \cap S \ne \emptyset\}$ has $\dim(S)$ elements. Let $J = \{j_1,\dots, j_{n - \dim(S)}\}$ be its complement. Then:
\begin{equation} \label{want1}
		\MV(\pi_S(\tl D_{j_1}) (\xi),\dots,\pi_S(\tl D_{j_{n-\dim(S)}})(\xi)) =  \Vol_\Z(\tl {P}_S(\xi)).
		\end{equation}
\end{lemma}

\subsection{Convex-geometrical proof of the main theorem}\label{convgeom_pr_sec}

This subsection proves formula (\ref{semieq}) from Theorem \ref{semith} using convex geometry. The argument relies on Lemma \ref{near}, which is proved separately in \S \ref{lem-subsec}.

We show that Theorem \ref{semith} reduces to Lemma \ref{near}. By Lemma \ref{hovamv}, we have:
	
\begin{equation}\label{le}
\mathfrak v (P) = \mathfrak v^\dagger (P) + \sum\limits_{\xi \in (\R^n)^* \setminus \{0\}} \MV(\tl D_1(\xi), \dots, \tl D_n(\xi)).
\end{equation}
	
Consider the sum on the right-hand side of (\ref{le}). For $\xi \in (\R^n)^* \setminus \{0\}$, if $P^\xi$ is not a suture, then there exist distinct indices $i_1, \dots, i_{\dim(P^\xi) + 1}$ such that $D_{i_1}\cap P^\xi \ne \emptyset, \dots, D_{i_{\dim P^\xi + 1}}\cap P^\xi \ne \emptyset$, and hence all $\widetilde{D}_{i_1}(\xi), \dots, \widetilde{D}_{i_{\dim(P^\xi) + 1}}(\xi)$ are contained in $P^\xi$. Thus, $\dim(P^\xi) + 1$ of the polytopes $\widetilde{D}_1(\xi), \dots, \widetilde{D}_n(\xi)$ lie in an affine space of dimension $\dim(P^\xi)$. By Proposition \ref{mv=0_prop}, their mixed volume is zero. Hence:
\begin{equation} \label{rasp}
	\sum\limits_{\xi \in (\R^n)^*\setminus\{0\}} \MV(\tl D_1(\xi), \dots, \tl D_n(\xi)) = \sum\limits_{S\in \Theta (P)\setminus\{P\}}\ \sum\limits_{\xi \in (\R^n)^*: P^\xi = S} \MV(\tl D_1(\xi), \dots, \tl D_n(\xi)).
\end{equation}  
Now let $S \ne P$ be a suture and $\xi \in (\R^n)^*$ with $P^\xi = S$. Recall the set that $I = \{i \mid D_i \cap S \ne \emptyset\}$ has $\dim(S)$ elements; denote them by $i_1,\dots, i_{\dim S}$. Note that for all $i\in I$ we have $\tl D_i(\xi) = D_i \cap S$. By Proposition \ref{mv=prod_rem}, the mixed volume factors as:
	\begin{equation}
	\MV(\tl D_1(\xi), \dots, \tl D_n(\xi)) = \MV(\tl D_{i_1}(\xi),\dots, \tl D_{i_{\dim(S)}}(\xi)) \cdot \MV(\pi_S(\tl D_{j_1}(\xi)),\dots, \pi_S(\tl D_{j_{n-\dim(S)}}(\xi))).
	\end{equation}
For a suture $S$ denote by $\mathbf S$ the set $S \cap \Ps$. The first factor is $\mathfrak{v}^\dagger(S)$ by definition, and the second is given by the main Lemma \ref{near}. Thus:
	\begin{equation} \label{no_dot}
	\sum\limits_{\xi \in (\R^n)^*: P^\xi = S} \MV(\tl D_1(\xi), \dots, \tl D_n(\xi)) = \mathfrak{v}^\dagger(S)  \sum\limits_{\xi \in (\R^n)^*: P^\xi = S} \Vol_\Z(\tl  P_S(\xi)) = \mathfrak{v}^\dagger (S) \cdot  c_\Ps^{\mathbf S}.
	\end{equation}
Substituting this into (\ref{le}) yields:
	
	\begin{equation} \label{rec}
	\mathfrak v (P) = \mathfrak v^\dagger (P) + \sum\limits_{S \in \Theta (P)\setminus \{P\}} \mathfrak v^\dagger (S) \cdot c_\Ps^{\mathbf S}.
	\end{equation}
For $S \in \Theta(P)$, let $\Theta(S)$ be the set of sutures contained in $S$. By Remark \ref{semi_ind_rem}, we obtain the system:
	\begin{equation} \label{recS}
	\mathfrak v (S)  = \mathfrak v^\dagger (S) + \sum\limits_{S^\prime \in \Theta(S) \setminus S} \mathfrak v^\dagger (S^\prime) \cdot c_{\mathbf S}^{\mathbf{S^\prime}}
	\tag{\ref{rec}$_S$}.
	\end{equation}
The system (\ref{recS}) for all $S \in \Theta (P)$ is equivalent to the matrix identity (\ref{semieq}). Therefore, Theorem \ref{semith} follows from the main Lemma \ref{near}.

\subsection{Proof of the main lemma} \label{lem-subsec}

To prove Lemma \ref{near}, we first establish an auxiliary result. 

\begin{lemma} \label{formal}
Let $D$ be a daughter polytope of $\Ps$ with $\Df = \Df(D)$. Consider a face $W \in \Df$ and a face $S \subset W$. Then for any other face $W' \in \Df \setminus \{W\}$, the projection $\pi_S(W')$ does not intersect $\Ps_S \subset \pi_S(\Z^n)$.
\end{lemma}
	
\begin{proof}
Suppose, for contradiction, that there exists $W' \in \Df \setminus \{W\}$ such that $\pi_S(W')$ contains a point $X \in \Ps$ with $\pi_S(X) \in \Ps_S$. We show this forces $W = W'$.
    
If $\mathbf{O}_S \in \pi_S(W')$, then $W'$ intersects $S$, implying $W' = W$ by the definition of daughter polytopes.
		
If $\mathbf O_S \notin \pi_S(W^\prime)$, we construct a face $F$ of $D$ such that $F \cap W \ne \emptyset$, $F \cap W^\prime \ne \emptyset$, and  $F \cap D = \emptyset$.  This contradicts the definition of daughter polytopes unless $W = W'$.

Choose a covector $\gamma_S \in (\pi_S(\R^n))^*$ satisfying the following conditions (where $H(\gamma_S ;\bullet)$ is the support function from Definition \ref{support_face_def}):
		
\begin{enumerate}
	\item [a1.] $(\pi_S(P))^{\gamma_S} = \mathbf O_S$.
	\item [a2.] $\pi_S(X) \in \Conv(\pi_S(P) \setminus \mathbf O_S) ^ {\gamma_S}$.
	\item [a3.] $H (\gamma_S ; \pi_S(X)) - H (\gamma_S; \mathbf O_S) = - 1$.
\end{enumerate}		
Since $W'$ is a face of $P$ and $W' \cap S = \emptyset$ (by definition of daughter polytopes), there exists $\gamma_{W'} \in (\R^n)^*$ such that:
\begin{enumerate}
	\item [b1.] $P ^{\gamma_{W^\prime}} = W^\prime$.
	\item [b2.] $H (\gamma_{W^\prime}; S) - H (\gamma_{W^\prime} ; W^\prime) = - 1$.
\end{enumerate}
		
\begin{sign}\label{triv_prol}
Let $V_2 \subset V_1$ be affine spaces. Any covector $\gamma \in (V_1/V_2)^*$ extends to a covector in $V_1^*$ that vanishes on $V_2$. By slight abuse of notation, we denote this extension also by $\gamma$.
\end{sign}
	
Define $F = P^{\gamma_{W'} + \gamma_S}$. We show that $(\Ps \cap F) \subset W' \cup S$, with $W' \cap F \ne \emptyset$ and $S \cap F \ne \emptyset$, by proving:

\begin{enumerate}
	\item $ H(\gamma_{W^\prime} + \gamma_S; X ) = H (\gamma_{W^\prime} + \gamma_S; S)$.
			
	Indeed, $H(\gamma_{W^\prime} + \gamma_S; S) = H(\gamma_{W^\prime}; S)  + H(\gamma_S ; S)$ since $\gamma_S$ is constant on $S$. We have $H(\gamma_{W^\prime} + \gamma_S; X) = H(\gamma_{W^\prime}; X) + H(\gamma_S;X)$ since $X$ is a point. Thus, the equality follows from a3. and b2.
			
	\item For all $Y \in \Ps \setminus (S \cup W')$, we have $H(\gamma_{W'} + \gamma_S; Y) < H(\gamma_{W'} + \gamma_S; X)$.
    
	From a2., $H(\gamma_S; X) \ge H( \gamma_S ; Y)$ since $Y \notin S$. From b1., $H (\gamma_{W'}; X) > H(\gamma_{W'}; Y)$ since $Y \notin W'$. Thus, the inequality holds.
\end{enumerate}
		
Therefore, $F \cap D = \emptyset$ by definition of daughter polytopes, yet $F \cap W' \ne \emptyset$ and $F \cap S \ne \emptyset$ (so $F \cap W \ne \emptyset$). This implies $W' = W$, a contradiction.
		
\end{proof}
	
Now we prove the main lemma (Lemma \ref{near}).
	
\begin{proof} [Proof of Lemma \ref{near}]

Note that the face $\widetilde{P}_S(\xi)$ is either the point $\mathbf{O}_S$, the face $(\Conv(\Ps_S \setminus \mathbf{O}_S))^\xi$, or the convex hull of these two. If $\widetilde{P}_S(\xi) = \mathbf{O}_S$, then the lemma is trivial, since both sides of (\ref{want1}) vanish. Denote $\mathbf{Y} = (\Ps_S \setminus \mathbf{O}_S)^\xi$.

\begin{enumerate}
\item  Case  $\tl{P}_S (\xi) = \Conv(\O, \mathbf Y)$:
 
For each $j \in J$, let $W_j \in \Df_j$ be the face of $P$ containing $S$. By Lemma \ref{formal} we have: $$\pi_S(\widetilde{D}_j)(\xi) = \Conv(\mathbf{O}_S \cup (\mathbf{Y} \setminus \pi_S(W_j))).$$ 
We show that the polytopes $\pi_S(\widetilde{D}_j)(\xi)$ for $j \in J$ are interlaced in $\widetilde{P}_S(\xi)$.

Suppose, for contradiction, that they are not interlaced in $\widetilde{P}_S(\xi)$. Then there exists a face $K \subsetneq \widetilde{P}_S(\xi)$ such that the set $J^K = \{j \in J \mid \pi_S(\widetilde{D}_j)(\xi) \cap K = \emptyset\}$ satisfies $|J^K| \ge |J| - \dim(K)$.

%Значит, достаточно доказать, что с каждой гранью $K \subsetneq \tl \Ap_S^\xi$ пересекаются хотя бы $dim(K) + 1$ многогранников среди $p_S(\tl\Ap_{j})^\xi, j \in J$.

If $\mathbf{O}_S \in K$, then all $\pi_S(\widetilde{D}_j)(\xi)$ intersect $K$, a contradiction. So assume $\mathbf{O}_S \notin K$. Let $F$ be the minimal face of $\pi_S(P)$ containing $\mathbf{O}_S$ and $K$. Choose a covector $\gamma \in (\pi_S(\R^n))^*$ such that $\pi_S(P)^\gamma = F$. Extend $\gamma$ to $\R^n$ (see Notation \ref{triv_prol}).
			
For each $j \in J^K$, the face $P^\gamma$ is contained in $W_j$ (since $S \subset W_j$ and $F \subset \pi_S(W_j) $), so $P^\gamma \cap D_j = \emptyset$. Thus, at least $|J^K| \ge n - \dim(S) - \dim(K)$ daughter polytopes do not intersect $P^\gamma$, meaning at most $\dim(S) + \dim(K)$ intersect it. But $\dim(P^\gamma) \ge \dim(K) + \dim(S) + 1$, contradicting the semi-interlaced condition.
			
Hence, the polytopes are interlaced in $\widetilde{P}_S(\xi)$, and (\ref{want1}) follows from Proposition \ref{interlaced_prop}.

\item Case $\tl{P}_S (\xi) = \Conv (\mathbf Y)$:

Here $\Vol_\Z(\widetilde{P}_S(\xi)) = 0$. We show that at least $\dim(\mathbf{Y}) + 1$ of the polytopes $\pi_S(\widetilde{D}_j)(\xi)$ ($j \in J$) are contained in $\Conv(\mathbf{Y})$, implying both sides of (\ref{want1}) vanish by Proposition \ref{mv=0_prop}.

Choose $c \in (1, \infty)$ such that $\widetilde{P}_S(c\xi) = \Conv(\mathbf{O}_S \cup \mathbf{Y})$. By the previous case, there exists $J^c \subset J$ with $|J^c| \ge \dim(\mathbf{Y}) + 1$ such that for all $j \in J^c$, $\pi_S(\widetilde{D}_j)(c\xi) \cap \Conv(\mathbf{Y}) \ne \emptyset$. Hence, for all $j \in J^c$, we have $\pi_S(\widetilde{D}_j)(c\xi) = \Conv (\pi_S(\widetilde{D}_j)(\xi), \mathbf{O}_S)$ and $\pi_S(\widetilde{D}_j)(\xi) = \pi_S(D_j^\xi) \subset \Conv(\mathbf{Y})$.

\end{enumerate}

\end{proof}

The following corollary is implied by Case 2 of the proof of the main Lemma \ref{near}.

\begin{cor}\label{for_alg_non_deg_cor}

Let $D_1, \dots, D_n$ be semi-interlaced in $\Ps \subset \Z^n$, and let $S$ be a suture. For any $\xi \in (\pi_S(\R^n))^*$ with $\pi_S(P)^\xi = \mathbf{O}_S$, at least $\dim((\Ps_S \setminus \mathbf{O}_S)^\xi) + 1$ of the polytopes $\pi_S(D_{j_1}), \dots, \pi_S(D_{j_{n - \dim S}})$ intersect $(\Ps_S \setminus \mathbf{O}_S)^\xi$.

\end{cor}

\subsection{Algebro-geometric proof of the main theorem} \label{algem_pr_sec}

%Note that the coefficients $c_\bullet^\bullet$ are closely related to Euler obstructions of toric varieties (see \cite{MT08} and \cite{Es08}) and the Proposition \ref{orbit_mult} below follows from Lemma 1.28 from \cite{Es08}.

%\begin{proposition}\label{orbit_mult}
 %   If for a particular (not generic) system (\ref{basic}) with supports contained in $\Ps$ there is a root in an orbit $X_F$ corresponding to a face $F \subsetneq P$  then the root's multiplicity is greater or equal to $c_P^F$ and it is equal to $c_P^F$ for nondegenerate system.
%\end{proposition}

This subsection presents proofs of the mixed volume formulas for interlaced and semi-interlaced polytopes (Proposition \ref{interlaced_prop} and Theorem \ref{semith}) using projective toric varieties recalled in \S \ref{pr_tor_subsec} and \S \ref{mult_tor_var_subsec}. The proof for semi-interlaced polytopes relies on Corollary \ref{for_alg_non_deg_cor}.

\begin{proof}[Proof of the formula for interlaced polytopes (Proposition \ref{interlaced_prop}) via projective toric varieties] \ \\
Let $D_1, \dots, D_n$ be lattice polytopes interlaced in $P = \Conv(D_1 \cup \dots \cup D_n) \subset \R^n$, and let $\Ps$ be a lattice set such that $P = \Conv(\Ps)$ and all vertices of each $D_i$ lie in $\Ps$. Consider a generic system of Laurent polynomials with supports $D_1 \cap \Ps, \dots, D_n \cap \Ps$. These polynomials define hyperplane sections of the projective toric variety $X_\Ps$. Consider any proper face $F \subsetneq P$. By the definition of interlaced polytopes, at least $\dim(F) + 1$ of the polytopes $D_1, \dots, D_n$ intersect $F$; therefore Corollary \ref{gen_sys_no_sol_rem} implies that the system has no solutions in the open orbit of $X_{\Ps \cap F}$. Thus, there are no solutions on orbits of positive codimension, and all solutions lie in the open torus. This yields $\MV(D_1, \dots, D_n) = \Vol_\Z(P)$.
\end{proof}

\begin{proof}[Proof of the formula for semi-interlaced polytopes (Theorem \ref{semith}) via projective toric varieties] \ \\
Let $D_1, \dots, D_n$ be semi-interlaced polytopes in a finite set $\Ps \subset \Z^n$, and consider generic polynomials with supports $D_1 \cap \Ps, \dots, D_n \cap \Ps$. These polynomials define hyperplane sections of the projective toric variety $X_\Ps$.

For a face $\mathbf{S}$ of $\Ps$, by Corollary \ref{gen_sys_no_sol_rem}, the system can have solutions in the open orbit of $X_{\mathbf{S}}$ only if $S = \Conv(\mathbf{S})$ is a suture. The number of solutions in this open orbit, by Remark \ref{BK_toric_rem}, is $$\frac{1}{\ind_{\Z^n \cap \aff(\mathbf{S})}} \MV(\{D_i \cap S \mid D_i \cap S \ne \emptyset\}) = \frac{\mathfrak{v}^\dagger(S)}{\ind_{\Z^n \cap \aff(\mathbf{S})}(\mathbf{S})}.$$ By Corollary \ref{for_alg_non_deg_cor} and Corollary \ref{gen_sys_no_sol_rem}, all solutions in the open orbit of $X_{\mathbf{S}}$ are nondegenerate in the sense of Lemma \ref{non_deg_cond_rem}. Thus, by Theorem \ref{toric_orb_mult_th}, the total multiplicity of solutions in this orbit is $$\frac{\mathfrak{v}^\dagger(S) \cdot c_{\Ps}^{\mathbf{S}}}{\ind_{\Z^n}(\Ps)}.$$

By Remark \ref{BK_toric_rem}, the total number of solutions in $X_\Ps$, counted with multiplicities, is $\frac{\Vol_\Z(P)}{\ind_{\Z^n}(\Ps)}$. This yields equation (\ref{rec}) with both sides divided by $\ind_{\Z^n}(\Ps)$. Repeating this argument for all sutures $S$ gives the system (\ref{recS}), which implies the desired formula (\ref{semieq}) for the mixed volume of semi-interlaced polytopes.

\end{proof}

\section{Examples and applications of semi-interlaced polytopes}

In this section, we introduce off-coordinate polytopes as a special case of semi-interlaced polytopes. We then show that, in the Newton-nondegenerate setting, the mixed volume of off-coordinate polytopes appears in the global Kouchnirenko formula and determines several algebraic degrees. Motivated by Arnold's monotonicity problem, we study conditions under which the mixed volume of off-coordinate polytopes vanishes.

%In this section we introduce special cases of semi-interlaced polytopes, called off-coordinate polytopes. We show that in the Newton nondegenerate case the mixed volume of certain off-coordinate polytopes participates in the computation of the global Kouchnirenko formula and various algebraic degrees. Motivated by Arnold's monotonicity problem, we study conditions under which the mixed volume of off-coordinate polytopes vanishes.

%This section introduces a concrete family of semi-interlaced polytopes, called off-coordinate polytopes. We show that their mixed volume computes the Kouchnirenko formula for hypersurface singularities and various algebraic degrees in the Newton nondegenerate case. Motivated by Arnold's monotonicity problem, we study conditions under which the mixed volume of off-coordinate polytopes vanishes.

\subsection{Off-coordinate polytopes} \label{off_sec}

Our main examples of semi-interlaced polytopes are provided by off-coordinate polytopes.

\begin{definition}\label{V_face_def}
Let $P$ be a polytope and in a cone $C$. A face $F$ of $P$ is called a \textbf{V-face} if there exists a face $Q$ of $C$ such that $F \subset Q$ and $\dim F = \dim Q$.
\end{definition}

\begin{definition}[cf. Definition \ref{off_def}]\label{off_glob_def}
Let $C = C_\D \oplus \R^{n-m}$ , where $C_\D$ is a simplicial cone with $\dim C_\D = m$, and let $\Ps \subset C$ be a finite lattice set with $P = \Conv (\Ps)$. Denote by $E_1, \dots, E_m$ the facets of $C$. The \textbf{off-coordinate} polytopes are defined as:
$$D_i = \begin{cases}
        \Conv(\Ps \setminus E_i) &  1 \le i \le m,\\
        \Conv(\Ps) & m < i \le n.
\end{cases}$$
We denote the mixed volume of the off-coordinate polytopes by $\Voff_C(\Ps)$.
\end{definition}

\begin{Remark}\label{off_rem}
Each $D_i$ is indeed a daughter polytope of $\Ps$ (the corresponding set $\mathcal{D}_i$ is either a single face $P \cap E_i$ or empty). The number of off-coordinate polytopes intersecting a face $F \subset P$ equals the dimension of the minimal face of $C$ containing $F$. Therefore, the off-coordinate polytopes are indeed semi-interlaced, and the sutures are precisely the V-faces of $P$.    
\end{Remark}

\subsection{Global Kouchnirenko formula and negligible polytopes}

This subsection recalls the global Kouchnirenko formula and the notion of negligible polytopes in $\R^n_{\ge 0}$. We show that the Kouchnirenko formula equals the mixed volume of off-coordinate polytopes and that these polytopes provide an alternative method to verify the negligibility of so called $B_k$-polytopes.

\begin{definition} \label{Newton_def}
A lattice polytope $P \subset \R^n_{\ge 0}$ is called \textbf{convenient} if it contains the origin $O$ and all standard basis vectors. The \textbf{Newton number} of a convenient polytope $P$ is defined as the alternating sum
    $$\nu(P) = \sum_{E \subset \R^n} (-1)^{n-\dim E} \cdot \Vol_\Z(P \cap E), $$
taken over all coordinate subspaces $E$ of $\R^n$. A convenient polytope in $\R^n_{\ge 0}$ is called \textbf{negligible} if $\nu(P) = 0$.
\end{definition} 

We recall the global Kouchnirenko theorem:

\begin{theorem}(\cite{Kou76} and \cite[Proposition 3.3]{Bro88})
Let $f \in \C[z_1,\dots, z_n]$ be a polynomial with a convenient Newton polytope $P \subset \R^n_{\ge 0}$ and generic coefficients. Then the hypersurface $\{f=0\} \subset \C^n$ has the homotopy type of a wedge of $\nu(P)$ of $(n-1)$-dimensional spheres.
\end{theorem}

The Newton number can be expressed as the mixed volume of off-coordinate polytopes. Let $\Ds_n \subset \Z_{\ge 0}^n$ denote the set of $n+1$ vertices of the standard simplex in $\Z_{\ge 0}^n$.

\begin{Remark}\label{nu=voff_rem}
Let $P \subset \R^n_{\ge 0}$ be a convenient polytope, and let $\Ps$ be a lattice set such that $\Conv(\Ps) = P$ and $\Ds_n \subset \Ps$ (e.g., $\Ps = P \cap \Z^n$). By Theorem \ref{semith} and Remark \ref{+-1_rem}, the Newton number satisfies:
$$\nu(P) = \Voff_{\R^n_{\ge 0}} (\Ps).$$
\end{Remark}

An important case, particularly in the context of Arnold's monotonicity problem (\cite[1982-16]{Arn04}; see, e.g., \cite{Sel24} and \cite{Sel25} for a more up-to-date version), is when the Newton number vanishes. The classifications of such polytopes, called $\mathbf{B_k}$\textbf{-polytopes}, is given in \cite[Theorem 1.3]{Sel25}.

\begin{definition} \label{Cayley_sum_def}
    Let $k \le n$ be nonnegative integers. Let $e_1, \dots ,e_k$ be the standard basis of $\Z^k$. For finite sets $\Ps_0, \Ps_1,\dots, \Ps_k \subset \Z^{n-k}$, the \textbf{Cayley sum} is defined as
    $$\Ps_0 \ast \Ps_1 \ast \cdots \ast  \Ps_k = (\Ps_0 \times \{0\}) \cup (\Ps_1 \times \{e_1\}) \cup \dots \cup (\Ps_k \times  \{e_k\}).$$
    The Cayley sum of convex polytopes is defined as the convex hull of the corresponding shifts of polytopes.
\end{definition}

\begin{definition}\label{B_k_def}
Let $P_0, P_1, \dots, P_k \subset \R^{n-k}$ be lattice polytopes such that
$$\dim (P_1 + \dots + P_k) < k \quad \text{and} \quad \dim (P_0) = n-k.$$
Then the Cayley sum $P_0 \ast P_1 \ast \cdots \ast  P_k \subset  \R^{n-k} \oplus \R^k_{\ge 0}$ is called a \textbf{$\mathbf{B_k}$-polytope} based on $\R^k_{\ge 0}$.
\end{definition}

The following result is a special case of \cite[Theorem 1.3]{Sel25}:

\begin{theorem} (cf. \cite[Theorem 1.3]{Sel25})\label{neg_class_th}
A convenient polytope $P$ in $\R^n_{\ge 0}$ is negligible (i.e., $\nu(P) = 0$) if and only if it is a $B_k$-polytope based on a coordinate subspace $\R^k_{\ge 0} \subset \R^n_{\ge 0}$.
\end{theorem}

Let us verify it in one direction using the off-coordinate polytopes.

\begin{proof}[Proof: $B_k$-polytopes are negligible in $\R^n_{\ge 0}$]

Let $P = P_0 \ast \dots \ast P_k \subset \R^n_{\ge 0}$ be a convenient $B_k$-polytope. Define the finite set
$$\Ps_0 \ast \Ps_1 \ast \cdots \ast  \Ps_k \subset  \Z^n_{\ge 0},$$
where $\Ps_i = P_i \cap \Z^{n-k}$. By Remark \ref{nu=voff_rem}, we have 
$\nu(P) = \Voff_{\R^n_{\ge 0} } (\Ps)$.
Among the off-coordinate polytopes there are $P_1, \dots, P_k$, and since $\dim(P_1 + \dots + P_k) < k$, Proposition \ref{mv=0_prop} implies that the mixed volume vanishes. Hence, $\nu(P) = 0$.
    
\end{proof}

The proof in the converse direction in \cite{Sel25} uses a non-trivial technique involving dual defective sets, specifically the Furukawa–Ito classification \cite{FI21}.

\begin{Que}\label{neg_comb_que}
By Remark \ref{nu=voff_rem} and Theorem \ref{neg_class_th}, a convenient polytope $P \subset \R^n_{\ge 0}$ is a $B_k$-polytope if and only if $\Voff_{\R^n_{\ge 0}}(P \cap \Z^n) = 0$. Can this equivalence be proved using combinatorial properties of mixed volumes?
\end{Que}

\subsection{Off-coordinate polytopes with zero mixed volume} \label{voff=0_subsec}

In addressing Question \ref{neg_comb_que}, we are led to the following problem. %Example \ref{voff=0_not_bk} demonstrates that Question \ref{neg_comb_que} is nontrivial.

\begin{Que}\label{voff_2=0_que}
Characterize lattice sets $\Ps \subset \R^n_{\ge 0}$ with convenient convex hulls $P = \Conv(\Ps)$ for which $\Voff_{\R^n_{\ge 0}}(\Ps) = 0$.
\end{Que}

We present examples of sets $\Ps \subset \R^n_{\ge 0}$ with $\Voff_{\R^n_{\ge 0}}(\Ps) = 0$ that generalize $B_k$-polytopes.

\begin{ex}[Stretched $B_k$-sets]\label{Voff=0_bk}
   Let $k \le n$ be non-negative integers. Let $e_1, \dots ,e_k$ be the standard basis of $\Z^k$, and $a_1,\dots, a_k$ be positive integers. For finite lattice sets $\Ps_0, \Ps_1,\dots, \Ps_k \subset \Z^{n-k}$, the \textbf{stretched Cayley sum} is defined as:
   $$\Ps_0 \ast_{a_1} \Ps_1 \ast_{a_2} \cdots \ast_{a_k}  \Ps_k = (\Ps_0 \times \{0\}) \cup (\Ps_1 \times \{a_1 e_1\}) \cup \dots \cup (\Ps_k \times  \{a_k e_k\}).$$
    
\begin{definition}\label{stretched_bk_def}
Let $\Ps_0, \Ps_1, ..., \Ps_k \subset \Z^{n-k}$ be lattice sets such that
$$\dim (\Ps_1 + \dots + \Ps_k) < k \quad \text{and} \quad \dim (\Ps_0) = n-k.$$
Then the stretched Cayley sum $\Ps_0 \ast_{a_1} \Ps_1 \ast_{a_2} \cdots \ast_{a_k}  \Ps_k \subset \Z^k \oplus \Z^{n-k}$ is called a \textbf{stretched} $\mathbf{B_k}$\textbf{-set}.
\end{definition}
Note that if $\Ps \subset \R^n_{\ge 0}$ is a stretched $B_k$-set, then $\Voff_{\R^n_{\ge 0}}(\Ps) = 0$ by Proposition \ref{mv=0_prop}.\\
\end{ex}

The above construction does not provide a complete answer to Question \ref{voff_2=0_que}. Here is a simple example of a different nature.

\begin{ex}\label{voff=0_not_bk}
Consider $\Z^4$ with coordinates $e_1, e_2, e_3$, and $e_4$. Consider the set
$$\Ps = \{0, 2e_1,2e_2, e_1 + e_2, e_3\}\times  \{0,e_4\}.$$ 

\begin{figure}[H]
		\begin{center}
	\includegraphics[scale=.4]{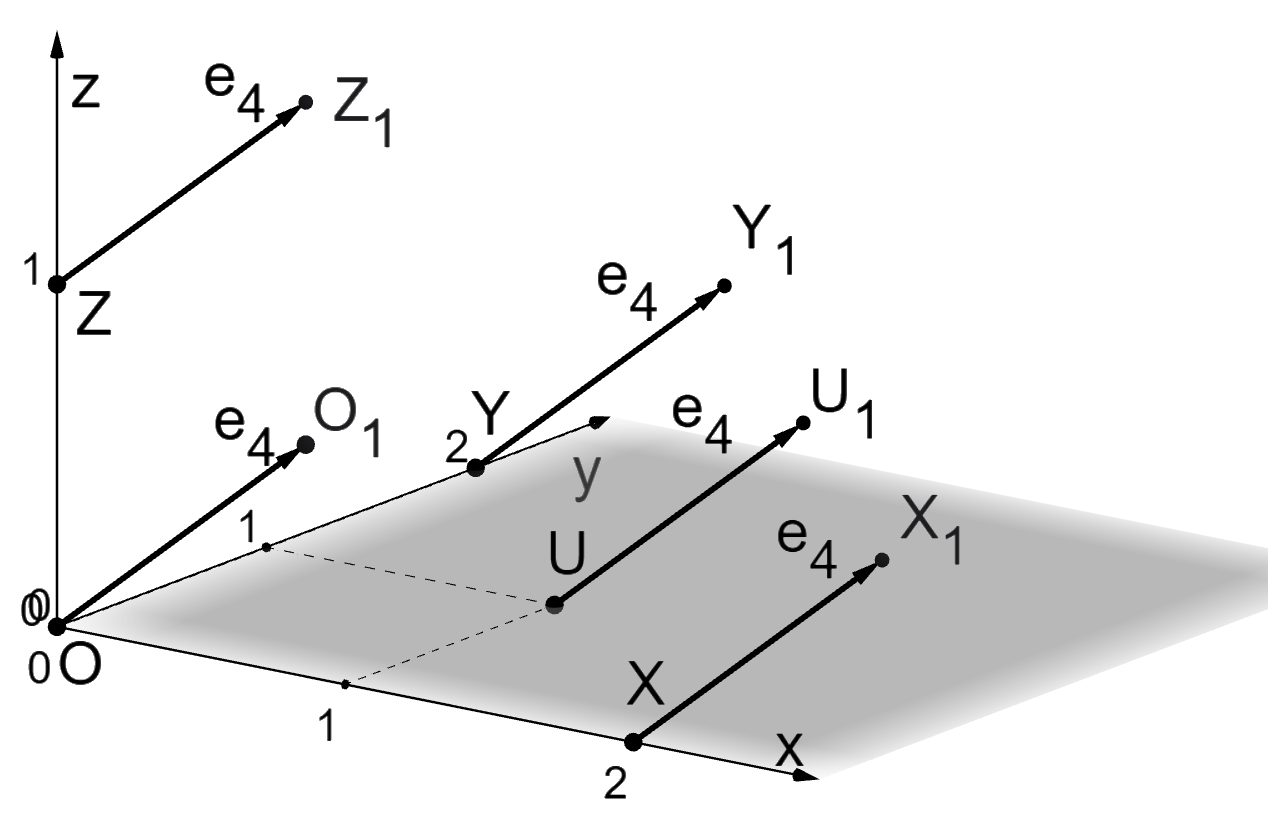}
		\end{center}
		\caption{
		\label{Voff=0_ex} Example of $\Voff_{\R^4_{\ge 0}}(\Ps) = 0$ that is not a stretched $B_k$-set.}
\end{figure}
By Proposition \ref{mv=0_prop}, we have: 
$$\dim (D_1 + D_2 + D_3) = 2 \quad \Rightarrow  \quad  \Voff_{\R^4_{\ge 0}}(\Ps) = 0.$$
%Note that the segment $[2e_1, 2e_2]$ in this example can be replaced with any segment in the $(e_1,e_2)$-coordinate plane. 

\end{ex}

\subsection{Algebraic degrees}\label{alg_deg_subsec}

The mixed volume of off-coordinate polytopes computes various algebraic degrees (maximum likelihood, Euclidean distance and polar degrees) in the Newton nondegenerate case. For background on these degrees, we refer to \cite{Huh13} (maximum likelihood), \cite{DHOST16} and \cite[Corollary 4.6]{Huh13} (Euclidean distance and polar degrees). A general introduction to algebraic degrees can be found in \cite{BKS}, and a concise treatment in our context is given in \cite[\S 5.5]{Sel25}.

The following remark adapts corresponding statements from \cite[Example 1.7(6)]{Es24} and \cite[Theorem 5.40]{Sel25}, as in these cases the formulas for the mixed volume of semi-interlaced polytopes coincide with those for the $e$-Newton number (cf. Remark \ref{0_cci_rem}).

\begin{Remark}

The mixed volume of off-coordinate polytopes computes algebraic degrees as follows:

\begin{enumerate}
    \item \textbf{Maximum likelihood degree (ML-degree)} \cite[Theorem 2.13]{LNRW23} and \cite[Example 5.44]{Sel25}.\\
    For Newton nondegenerate polynomials $f_1,\dots,f_m \in \C[z^{\pm 1}_1,\dots,z^{\pm 1}_{n-m}]$ with supports $\Ps_1,\dots,\Ps_m$ and generic $u\in \Z^{n-m}$, the ML-degree of the complete intersection $V_{f} = \{f_1 = {\dots} = f_m = 0\}$ is:
    $$\text{MLdeg} (V_{f},u) = \Voff_{\R^{m}_{\ge 0} \oplus \R^{n-m}}(\Ps),$$
    where $\Ps^u = \{u\}\ast \Ps_1\ast\dots\ast \Ps_m \subset \Z^{m}_{\ge 0} \oplus \Z^{n-m}$.
    
    \item \textbf{Euclidean distance degree (ED-degree)} \cite{BSW22}, \cite{TT24} and \cite[Example 5.47]{Sel25}.\\
    Denote by $\Ds^\rho$ the support of the quadratic distance function (containing the origin, all coordinate basis vectors, and their doubles).
    For Newton nondegenerate polynomials $f_1,\dots,f_m \in \C[z^{\pm 1}_1,\dots,z^{\pm 1}_{n-m}]$ all of whose supports $\Ps_1,\dots,\Ps_m$ contain the origin $O$, the ED-degree of the complete intersection $V_{f} = \{f_1 = {\dots} = f_m = 0\}$ is equal to:
    $$\text{EDdeg} (V_{f}) = \Voff_{\R^{m}_{\ge 0} \oplus \R^{n-m}} (\Ps^\rho),$$
    where $\Ps^\rho = \Ds^\rho\ast \Ps_1\ast\dots\ast \Ps_m \subset \Z^{m}_{\ge 0} \oplus \Z^{n-m}$. 
    
    \item \textbf{Polar degree (P-degree)} \cite[Proposition 5.51]{Sel25}.\\
    Denote by $\D^\circ \subset \R^{n+1}$ the $n$-dimensional lattice simplex $\{\sum_i x_i = 1, x_i\ge 0\}$ and by $\Ds^\circ \subset \Z_{\ge 0}^{n+1}$ the set of its vertices. Denote by $C_d\subset \R^{n+2}$ the $(n+1)$-dimensional cone containing $d\D^\circ \ast \D^\circ$ such that the simplices $\D^\circ$ and $d \D^\circ$ are its hyperplane sections. For a homogeneous degree $d$ polynomial $f$ with support $\Ps \subset d\D^\circ$ and Newton nondegenerate coefficients, the polar degree is:
    $$\text{Pdeg}(f) = \Voff_{C_d}(\Ps \ast \Ds^\circ).$$
\end{enumerate}
    
\end{Remark}

\end{document}